\documentclass[preprint,12pt]{elsarticle}

\usepackage{enumerate}
\usepackage{graphicx,graphics}
\usepackage{amsfonts}
\usepackage{amssymb}
\usepackage{amsthm}
\usepackage{amsmath}
\usepackage{mathrsfs,color}
\usepackage{listings}
\input{xy}
\xyoption{all}

\newtheorem{theorem}{Theorem}[section]

\newtheorem{corollary}[theorem]{Corollary}

\newtheorem{proposition}[theorem]{Proposition}
\newtheorem{example}[theorem]{Example}
\newtheorem{definition}[theorem]{Definition}

\usepackage{amssymb}

\journal{arXiv}

\begin{document}

\begin{frontmatter}



\title{Approximation of  almost diagonal non-linear maps by  lattice  Lipschitz  operators}


\author{Roger Arnau, Jose M. Calabuig, Ezgi Erdo\u{g}an, Enrique A. S\'{a}nchez P\'{e}rez}

%

\address{ Roger Arnau, Jose M. Calabuig,
        Enrique \ A.\ S\'{a}nchez P\'{e}rez (Corresponding author).
              Instituto Universitario de Matem\'atica Pura y Aplicada,
              Universitat Polit\`ecnica de Val\`encia, Camino de Vera s/n, 46022
              Valencia, Spain.  ararnnot@posgrado.upv.es, jmcalabu@mat.upv.es, easancpe@mat.upv.es}

\address{Ezgi Erdo\u{g}an.
              Department of Mathematics, Faculty of Art and Science, University of Marmara, 34722, Kad\i k\"{o}y, Istanbul, Turkey. ezgi.erdogan@marmara.edu.tr}

\begin{abstract}
Lattice Lipschitz operators define a new class of nonlinear Banach-lattice-valued maps that can be written as diagonal functions with respect to a certain basis. In the $n-$dimensional case, such a map can be represented as a vector of size $n$ of real-valued functions of one variable. In this paper we develop a method  to approximate almost diagonal maps by means of lattice Lipschitz operators. The proposed technique is based on the approximation properties and error bounds obtained for these operators, together with a pointwise version of the interpolation of McShane and Whitney extension maps that can be applied to almost diagonal functions. In order to get the desired approximation, it is necessary to previously obtain an approximation to the set of eigenvectors of the original function. We focus on the explicit computation of error formulas and on illustrative  examples to present our construction. 
\end{abstract}




\end{frontmatter}


\section{Introduction and notation}

Diagonalizable operators play a fundamental role in classical operator theory. Indeed, they have special geometric features that allow the mathematical treatment of properties in many fields of science, and can be used to approximate linear operators in infinite-dimensional Banach spaces. However, there is no analog for the case of Lipschitz maps in Euclidean spaces, although there are many situations in which both classes (linear and Lipschitz maps) can be naturally associated. 

Several authors have recently been interested in the determination and characterization of classes of operators for which the notion of diagonalization could be adapted. To do this, one must begin by giving a definition of eigenvector that can make sense in the specific context in which one wishes to investigate.  Nonlinear spectral theory is already a classical subject and has a long development, both from a theoretical and practical point of view (\cite{appell,dan,furi}).  For example, for the case of multilinear mappings \cite{dasilva1,dasilva2,erd1,mil}, polynomials \cite{mac} or Lipschitz maps \cite{erd2}. 
In particular, we have studied suitable versions of these notions for Lipschitz maps in \cite{arn,erd2}. 

Our aim in the present paper is to adapt for the case of lattice Lipschitz operators the classical  extension/representation of  diagonalizable linear operators  once a basis of eigenvectors is known.  The idea is to use the same representation pattern (but in an approximate form) for the case of functions that are approximable by lattice Lipschitz maps, which is a broad class of functions containing linear operators. As we will explain, our method is based on the use of the order in the Euclidean lattice, rather than linearity. 

Thus, the class of lattice Lipschitz operators has recently been introduced to fill the gap in the helpful definition of diagonalizable (linear) operators for the case of Lipschitz \cite{arn} maps. Using the results presented in the aforementioned paper, we explain here an approximation method for ``almost diagonalizable" Lipschitz operators, i.e., Lipschitz functions that can be approximately computed as lattice Lipschitz maps.
In this paper we  provide specific tools for the (approximate) representation of general Lipschitz maps by means of their eigenvectors using the vector-valued lattice versions of the classical McShane and Whitney formulas that were obtained in \cite{arn}.

Although some theoretical statements will also be presented, this paper is written from the point of view of applied mathematics, with the idea of providing an efficient algorithm for the approximation of Lipschitz maps with the requirements, as we have said, of being ``almost diagonal''.  An explanation of the ``initial class" of lattice Lipschitz operators, together with some clarifying examples, are given in Section \ref{sub1}.   General results on error bounds for such approximation can be found in \cite{arn}, but we will go further in this direction in  Section \ref{sec:approx_eigenvectors} of the present paper, which is mainly devoted to the approximate calculation of a ``sufficient" subset of eigenvectors. As will be explained, our method is based on the McShane-Whitney extension which is based on the order of the studied operator restricted to that subset of eigenvectors, as if it were a lattice Lipschitz operator. Once such a set is obtained, we present in Section \ref{sec:order} how to define a suitable order in Euclidean space that allows the use of McShane-Whitney expressions. Sections \ref{S4} and \ref{S4.3} are devoted to give the representation formulas and the final computational algorithms. In Section \ref{S5} we show an illustrative example, and in Section \ref{con} we give some final comments. The algorithm used to compute a set of approximate eigenvectors, programmed in R, is given in Appendix \ref{app:R_code}.

We will use the notation of general topology and mathematical analysis.
We will write  $\mathbb{R}$ for the  set of the real numbers endowed with its standard Euclidean distance. If $(M,d)$ and  $(D,\rho)$ are  metric spaces, we say that 
a map $T: M \to D$ is Lipschitz if there is a constant $K>0$ such that $\rho \big(T(a),T(b) \big) \le K \, d \big( a,b \big)$ for all $a,b \in D$ \cite{cobzas}. The Lipschitz constant of $T$ is the infimum of all such constants $K.$ 

We will center our attention on Lipschitz-type functions between Euclidean spaces $E$. Recall that the classical McShane-Whitney Theorem (\cite{cobzas,mc, wh}) establishes that, if $S$  is a  subset  of a metric space $(D,d)$ and $T : S \to   \mathbb{R}$ is a Lipschitz function with  Lipschitz constant $K$,  we can  always find an extension to $D$ with the same Lipschitz constant. 
 There are two classical ways of computing such an extension, that are provided by the formulas
$$
{{T}^M}(b)=\sup_{}\Big\{T(a)-K\,d(b,a): \, a\in S \Big\}, \quad b \in D, \quad \text{(McShane)}
$$

$$
 T^W(b) =\inf_{}\Big\{T(a)+K\, d (b,a): \, a\in S \Big\}, \quad b \in D, \quad \text{(Whitney)}.
$$
Suitable versions of such formulas, adapted for Lipschitz operators $T:E \to E,$ will be the main computational  tools in this paper. 

Given a Euclidean space $E$ of dimension $n$ and a basis in it,   we can always define an associated order in $E$ that is provided by the  coordinate-wise ordering of the vectors  (two vectors $x=(x_1,...,x_n), y=(y_1,...,y_n) \in E$ are ordered,
$x \le y,$ if  $x_i \le y_i$ for $i=1,...,n$). This gives, when the 2-norm of the coordinates is considered, a Banach lattice structure $(E, \| \cdot\|, \le).$ From this point of view, each vector represented by its coordinates in the chosen basis can be considered as  a function $f,$ $\{1,...,n\} \ni w \mapsto f(w)=x_w.$ The standard symbols $\vee$ and $\wedge$ are used for the maximum and the minimum of two (or several) vectors in the lattice, respectively.

\section{Lattice Lipschitz operators} \label{sub1}

As has been shown in \cite{arn}, the operators belonging to the special class of Lipschitz maps, that are called lattice Lipschitz operators,  have a relevant property which motivates the method we propose in this paper.
\\ 

\textbf{Property.} \textit{Each lattice Lipschitz operator $T$ can be written as the McShane extension (equivalently, the Whitney extension) of the restriction $T|_{R(\mathcal B)}$  of the operator itself to the set of the rays $R(\mathcal B)$ generated by a basis $\mathcal B$ of eigenvalues of $T.$}
\\

 Although we cannot expect all Lipschitz maps to behave like lattice Lipschitz operators, we can use this specific class as an approximation family. The idea is to find suitable extensions for each (approximated) eigenvector basis we can obtain for $E,$ and mix them in a common structure to find good approximations to such a Lipschitz operator. 

On the other hand, the concept of diagonalizable linear operator on finite dimensional Euclidean spaces can be extended to the setting of Lipschitz maps by means of the notion of lattice Lipschitz operator. Essentially, these are Lipschitz maps that can be written as combination of real valued Lipschitz functions that works independently in the directions of a given basis of the space. The main characterizations and results regarding this class or maps have been extensively studied in \cite{arn}. In what follows we recall the main technical definitions and results, as well as some illustrative examples. 

Let $E$ be a finite dimensional linear space $\mathbb R^n.$
Let $\mathcal B$ be a fixed basis of $E.$ We will consider the order $\le$ provided by the pointwise order of the coordinates of the vectors of $E$ in the basis $\mathcal B.$ As usual, using the coordinate representation given by $\mathcal B,$ every vector in $x=(\alpha_1,...,\alpha_n)  \in E$ can be understood as a function $x: \{1,...,n\} \to \mathbb R,$ $x(w)=\alpha_w, $ $w \in \{1,...,n\}= \Omega.$

In this setting, we recall the definition of our main tool.

\begin{definition} (Definition  1 in  \cite{arn})
Let $E_0$ be a subset of $E.$ A Lipschitz operator $T:E_0 \to E$ is lattice Lipschitz (with respect to the order $\le$ associated to the basis $\mathcal B$)  if  there is a  function $K:\Omega \to \mathbb R^+$ such that for every $x,y \in E_0,$
\begin{equation}
	\label{eq:lattice_lips_def}
	\big|T(x)-T(y)\big|(w) \le K(w) \big|x-y\big|(w), \quad w \in \Omega.
\end{equation}
The pointwise infimum of the functions $K$ in the inequality above is called the associate function to $T.$
\end{definition}

\vspace{0.5cm}

As we will see below, lattice Lipschitz operators can be identified with another class of functions  that allows a geometric description.
A Lipschitz  operator $T: E_0 \to E$ is \emph{diagonal} with respect to a basis $\mathcal B = \{ x_1, x_2, \ldots, x_n\}$ of $E = \mathbb R^n$ if there exist real functions  $f_i : \mathbb R \to \mathbb R$ for $1 \leq i \leq n$ such that
	\begin{equation}
		\label{eq:diagonal_representation}
		T\Big( \sum_{i=1}^n \alpha_i x_i \Big) = \sum_{i=1}^n f_i(\alpha_i) x_i, \qquad \alpha_1, \alpha_2, \ldots, \alpha_n \in \mathbb R.
	\end{equation}
We call the functions  $f_i$ the \emph{coordinate functions} of $T$ with respect to the basis $\mathcal B.$

\vspace{0.5cm}

\begin{example} \label{exdiag1}
Let us give an example of diagonal map.
Consider the function $S: B_{\mathbb R^2} \to \mathbb R^2$ given by
$$
S(x,y)= (x^2 + y^2, 2 xy ), \quad (x,y) \in B_{\mathbb R^2}.
$$
Both the coordinates in the domain $(x,y)$ and in the range are assumed to be with respect to the canonical basis of $\mathbb R^2.$
To find the eigenvectors and the eigenvalue functions, just notice that we have to find the vectors
$(x,y)$ such that $S(x,y)$ and $(x,y)$ are linearly dependent.  This means that
$$
\left|
\begin{array}{cc}
x^2+y^2 & 2 x y \\
x & y
\end{array}
\right|=0,
$$
that is, $y (x^2+ y^2)= 2 x^2 y$, what leads to the solutions
$$
y=0 \, \,\, \text{with all} \, \,\, x \in \mathbb R, \quad  \text{or} \quad y^2=x^2,
$$
and so, the set of eigenvectors is
$$
\{(t,0): |t| \le 1\} \cup \{ (t,t):  |t| \le 1/\sqrt 2\}  \cup \{ (t,-t): |t| \le 1/\sqrt 2\}.
$$
Consider the basis for $\mathbb R^2$ given by $\mathcal B= \{(1,1),(1,-1) \}.$ Then
we can consider the change or coordinates provided by $(x,y)= z(1,1)+ v(1,-1),$ what gives
$$
x=z+v, \quad y=z-v,
$$
and so 
$$
z= \frac{x+y}{2}, \quad v= \frac{x-y}{2}.
$$
In this new basis, we can write the map $S$ as
$$
S \big((z(1,1)+v(1,-1) \big)= S(x,y)= \big( (z+v)^2 + (z-v)^2, 2 (z+v)(z-v) \big)
$$
$$
=  \big( 2z^2+ 2v^2, 2 (z^2-v^2) \big) = 2 z \cdot z \, (1,1) + 2v \cdot v \,(1,-1).
$$
Consequently, $S$ is a diagonal map defined by the Lipschitz eigenvalue functions (that is, the functions that provide the eigenvalue associated to a given eigenvector) $e_1 \big( z(1,1) \big)= 2 z$ and  $e_2 \big( v(1,-1) \big)= 2 v.$
The coordinate functions are $f_1(z)= 2 z^2$ and $f_2(v)=2 v^2,$ that are also Lipschitz in the domain.
An easy computation taking into account that the function is defined on $B_{\mathbb R^2}$ gives that the associate function is $K(1)=K(2)=2.$

However, note that $S$ is not diagonal with respect to the basis $\mathcal D= \{(1,0), (1,1)\}.$ Indeed, to be diagonal would mean that
for the change of coordinates provided by $(x,y)= z(1,0)+ t(1,1),$ we should have a diagonal representation. But note that
$$
S(z(1,0)+ t(1,1))= S(x,y)= S(z+t,t) 
$$
$$
= \big( (z+t)^2+t^2, 2(z+t)t \big)= ( z^2 + 2 t^2+2zt, 2zt + 2 t^2).
$$
Since $( z^2 + 2 t^2+2zt, 2zt + 2 t^2)$ cannot be written as $s_1(z) (1,0) + s_2(t) (1,1),$ we have that $S$ is not diagonal with respect to $\mathcal D.$
\end{example}

\vspace{0.5cm}

Next result establishes the identification among diagonal and lattice Lipschitz operators.

\begin{theorem}  \label{coorwise_iff}  (Theorem 3 in  \cite{arn})
	A Lipschitz  operator  $T : E \to E$  is lattice Lipschitz (with respect to a certain order $\le$) with associate function $K : \Omega \to \mathbb R$ if and only if $T$ is diagonal with respect to the basis $\mathcal B$ (associated to the order $\le$) with   Lipschitz constant functions $K(i),$ $i=1,...,n.$
\end{theorem}

Note that, following Theorem \ref{coorwise_iff} and taking into account the last part of Example \ref{exdiag1}, we know that an operator can be diagonal (and therefore lattice Lipschitz) with respect to a given eigenvector basis $\mathcal B$, and not be diagonal with respect to another eigenvector basis $\mathcal D.$ In this case, using the theorem we obtain that the operator is not lattice Lipschitz with respect to the order defined by $\mathcal D.$

Although we do not use it in this paper, we should note that Theorem \ref{coorwise_iff}  allows a ``local'' version, i.e., it can be adapted for operators that are only defined on subsets of $E$ satisfying certain properties.  In its generality, this result opens the door to the approximation of Lipschitz operators by means of lattice Lipschitz extensions.  
Since having an eigenvector basis will play a central role in this, we will fix in the next section a method to deal with the (approximate) computation of an eigenvector basis in the case of Lipschitz operators. 
\begin{example}
The previous result can be directly applied to characterize lattice Lipschitz operators. Let us explain with an easy example how to do it. Take the
function $G: \mathbb R^2  \to \mathbb R^2$ given by the formula
$$
G(x,y)= \Big( x-y+ \frac{|y|}{1+|y|}, \frac{|y|}{1+|y|} \Big), \quad (x,y) \in \mathbb R^2,
$$
where $(x,y)$  are the coordinates with respect to the canonical basis, as well as the coordinates in the range. 

\begin{itemize}

\item[1)]
Let us compute the conditions for being eigenvectors of the map. The requirement (that has to be computed pointwise) is given by the equation
$$
\begin{vmatrix}
x-y+ \frac{|y|}{1+|y|} & x \\
 \frac{|y|}{1+|y|} & y \\
\end{vmatrix} 
= xy -y^2 +  \frac{y |y|}{1+|y|} - \frac{x |y|}{1+|y|} = 0.
$$
We directly get that this relation is satisfied for all vectors as $(x,0),$ $x \in \mathbb R,$ and
$$
x \Big(y- \frac{|y|}{1+|y|} \Big) = y \Big(y- \frac{|y|}{1+|y|} \Big),
$$
what gives $x=y.$ Therefore,  all vectors as $(x,0)$ and $(x,x),$ $x \in \mathbb R,$ are eigenvectors.

\item[2)]
Consider the basis $\mathcal B=\{ (1,0), (1,1) \},$ and  take  the coordinates of the vectors in this basis to be $(z,t).$
The  equation $(x,y)= z(1,0)+ t(1,1)$ gives the change of coordinates 
$$
x= z+t, \quad y=t.
$$
Therefore, the formula 
\begin{align*}
	G(z(1,0)+ t(1,1))
	& = G(z+t,t)= \Big( z+ \frac{|t|}{1+|t|}, \frac{|t|}{1+|t|} \Big) \\
	& = z(1,0) +  \frac{|t|}{t(1+|t|)} \, t \, (1,1),
\end{align*}
for $(z,t) \in \mathbb R^2,$ 
gives a diagonal representation for the function $G.$

\item[3)]
The functions $g_1(z)=z$ and $g_2(t)= |t|/(1+|t|)$ are real valued Lipschitz functions with constant  $1.$ It is obvious for $g_1.$ For $g_2,$ just note that for  all
$ t_1, t_2 \in \mathbb R,$
\begin{align*}
	\Big| \frac{|t_1|}{1+|t_1|} - \frac{|t_2|}{1+|t_2|} \Big| 
	& =	\Big|  \frac{|t_1|(1+|t_2|) }{(1+|t_1|)(1+|t_2|)} - \frac{|t_2|(1+|t_1|)}{(1+|t_1|)(1+|t_2|)} \Big| \\
	& = \frac{|t_1 - |t_2| }{(1+|t_1|)(1+|t_2|)} 
	\le \big| |t_1|- |t_2| \big| \le |t_1-t_2|,
\end{align*}
what means that the Lipschitz constant is less or equal to one, and doing $t_1 \to \infty$ and $t_2=0$ we see that this constant is $1.$
The associate function is then given by $K(1)=1$ and $K(2)=1.$
\end{itemize}

Summing up, all the vectors as $z(1,0)$ and $t(1,1)$ are eigenvectors, with eigenvalues
$e_1\big(z(1,0)\big)=1$ and $e_2\big(t(1,1)\big)= |t|/t(1+ |t|).$ Using Theorem  \ref{coorwise_iff}, we can say that $G$ is a lattice Lipschitz operator when the order is the one inherited by the basis $\mathcal B.$

\end{example}

%
%


\vspace{0.5cm}


\section{Estimates of the eigenvectors of a Lipschitz operator: the exponential distribution Monte-Carlo model}
\label{sec:approx_eigenvectors}


Following the arguments provided in the previous sections, we need to find sets of vectors that are approximately eigenvectors of any Lipschitz map we want to analyze. In this section we explain how to perform a statistical procedure to obtain such sets.
As we have seen, to know a basis of eigenvectors with respect to which a given Lipschitz map is  diagonal allows an easy representation of a lattice Lipschitz operator.
  However, we plan to establish an approximation procedure for a broad class of Lipschitz maps, so we cannot expect them to be lattice Lipschitz. We intend to use the same formulas that provide  representations of lattice Lipschitz maps, so we need to determine---at least, approximately--- the set  of eigenvectors of a given Lipschitz map in order to fix a convenient basis of the space, if possible.

To find a set of (approximated) eigenvectors we mix geometric and stochastic arguments.
Suppose that $E$ is a real Euclidean space, and recall that  its dual space can be directly identified with the original space $E$ and the dual action is the scalar product. Take a Lipschitz function $T:E \to E.$ 
Following the same ideas used in \cite[Section 2]{erd2}
 we define the \textit{diagonal value $\lambda_T(x)$} of $T(x)$ as the  real number that satisfies
$$
\lambda_T(x)= \big\langle T(x), \frac{x}{\|x\|^2} \big\rangle, \quad x \in X.
$$
We will simply write $\lambda(x)$ instead of $\lambda_T(x)$ if $T$ is fixed in the context.

Note that if $x$ is an eigenvector of $T$ with associated eigenvalue $\beta \in \mathbb R,$  the  real number $\lambda(x)$ defined as above gives the eigenvalue, that is,
$$
\lambda(x)= \big\langle T(x), \frac{x}{\|x\|^2} \big\rangle =\big\langle \beta \cdot x, \frac{x}{\|x\|^2} \big\rangle = \beta \frac{\langle x, x \rangle}{\|x\|^2} = \beta.
$$
The main property of the diagonal value is that it represents the optimal diagonal approximation to $T(x)$.  Furthermore, for the case of Euclidean spaces, the optimal diagonal approximation coincides with the projection of $T(x)$ on the unit vector in the direction of $x.$ We explicitly write this in  Proposition \ref{promin}, including the proof---a straightforward consequence of the Euclidean geometry---for the aim of completeness. 

Let us define before the \textit{diagonal projection error,}  $\varepsilon_T(x)(\alpha)$, as the function of $\alpha \in \mathbb R$ that represents the size of the difference between $T(x)$ and $\alpha \cdot x.$ After normalization, the formula for this quantity is given by
$$
\varepsilon_T(x)(\alpha)= \frac{1}{\|x\|} \big\| T(x) - \alpha \cdot x \big\|, \quad \alpha \in \mathbb R.
$$
We will write $\varepsilon_T(x)$ for the minimal value of the diagonal projection error  and, as in the case of $\lambda,$  $\varepsilon(x)$ if $T$ has been already fixed. 

\begin{proposition} \label{promin}
Let $X$ be a  Euclidean space, $T:X \to X$  a Lipschitz operator and $x \in X.$ Then the diagonal value $\lambda_T(x)=\lambda(x)$ is the real number that minimizes 
the diagonal  projection error. That is, the diagonal error is given by
$$
\varepsilon(x)= 
\varepsilon_T(x) \big(\lambda(x) \big)=  \frac{1}{\|x\|}  \, \Big\| T(x) - \lambda(x) \cdot x \Big\| =  \frac{1}{\|x\|} \, \min_{\alpha \in \mathbb R} \Big\| T(x) - \alpha \cdot x \Big\|.
$$
\end{proposition}
\begin{proof} It is given by a direct computation. Note that the solution of the equation

\begin{eqnarray*}
\frac{\partial \varepsilon_T(x)^2(\alpha)}{\partial \alpha}& =&  \frac{1}{\|x\|^2}  \,  \frac{\partial}{\partial \alpha} \Big( \big\langle T(x) - \alpha \cdot x, T(x) - \alpha \cdot x \big\rangle \Big)
\\
&=&  \frac{1}{\|x\|^2} \, \frac{\partial}{\partial \alpha}  \Big( \|T(x)\|^2 -2 \alpha \langle T(x),x \rangle + \alpha^2 \|x\|^2 \Big)\\
&=&-  \frac{2}{\|x\|^2} \, \langle T(x),x \rangle +  \frac{2}{\|x\|^2} \,\alpha \|x\|^2 =0,
\end{eqnarray*}
is given by $\alpha = \langle T(x),x \rangle/\|x\|^2= \lambda(x).$ Since 
$$
 \frac{\partial^2  \varepsilon_T(x)^2(\alpha)}{\partial \alpha^2} = 2  \frac{1}{\|x\|^2} \, \|x\|^2 =2 >0, 
 $$ 
 we get the result. 
\end{proof}

\begin{corollary}
For every point $x\in E,$ the error $\varepsilon(x)$ is given by the formula
$$
\varepsilon(x) = \sqrt{\frac{\|T(x)\|^2}{\|x\|^2}- \lambda(x)^2}.
$$
\end{corollary}
\begin{proof} It is a consequence of  Proposition \ref{promin} and a direct calculation involving the definition of $\lambda(x)$. Indeed, we have that
\begin{align*}
\varepsilon(x)^2 &= \frac{1}{\|x\|^2} \Big| \langle T(x)-\lambda(x) \cdot x ,T(x)-\lambda(x) \cdot x  \rangle\Big|
\\
&=
\frac{1}{\|x\|^2} \left(  \|T(x)\|^2 - 2 \lambda(x) \langle T(x), x \rangle + \lambda(x) \, \|x\|^2    \right)\\
&= \frac{\|T(x)\|^2}{\|x\|^2} - \lambda(x)^2.
\end{align*}
\end{proof}

Using this result, the  idea  is to perform  a Monte Carlo procedure to obtain a large enough set of eigenvectors of the Lipschitz operator $T.$ It must be said that most of the effort for the design of Monte Carlo methods for operator diagonalization comes from quantum physics and stochastic analysis \cite{huss,lee,will,will2}. In general, these approaches do not provide a fundamental framework to support a general methodology, since they focus on the actual computation of quantities with some physical or mathematical meaning. Consequently, we propose in what follow our own procedure based on the functions defined just above.\\

We fix a suitable uniform value $\epsilon$ to be accepted for the error committed when the operator is approximated by its diagonal value. We will use such a set for the computation of the McShane and Whitney representation formulas, as a substitute of the exact set of eigenvalues that could not be computable.  An example of such  situation coincides with the case in which the Lipschitz operator $T$ is the addition of a linear diagonalizable operator $L$ plus a perturbation $P$ with small norm, $T= L+P.$
\\

Therefore, we will base our method on a sampling procedure supported by probabilistic arguments using the normal distribution and following the next steps.

\begin{itemize}
	
	\item[(1)] We start by fixing a bounded set in which we will search for our approximate eigenvectors; if no additional information is known, we choose a product $P$ of intervals in $\mathbb R^n$ centered in $0;$ in case we know previously that the eigenvectors are located in a particular set  $M,$ we use it instead. 
	
	\item[(2)] We use the uniform distribution to sample a starting set of vectors $S_0.$ In case we have some previous knowledge on the set, we can introduce  Bayesian procedures  to fix a more accurated probability distribution $\Psi$ for doing the sampling.
	
	\item[(3)] We compute the functional $\varepsilon(x)$---where $\varepsilon(\cdot)$ is given by the error formula provided in Proposition \ref{promin}---for all $x \in S_0.$
	Now, we consider the set $S_1$ of all the points $x$ that are most similar to an eigenvector in terms of having small $\varepsilon(x)$. This can be done by selecting the points with smallest $\varepsilon(x)$---for example, $10 \%.$ of the points of $S_0$---.
	
	\item[(4)] Now, for every $s \in S_1$, we start an iterative process. Sample a fixed number of points around $s$ using a gaussian distribution centered on $s$ and with variance depending of the error $\varepsilon(x)$ formula
	\begin{equation*}
		\psi(x)= \frac{1}{(\tau \, \varepsilon(s) \, \pi)^{n/2}} \, \exp \left( - \frac{|| a - x ||^2}{\tau \varepsilon(s)} \right),
	\end{equation*}
	where $\tau$ is a fitting parameter.
	Select the point with the smallest $\varepsilon(\cdot)$ and repeat this step a fixed number of times---changing the value of $\tau$ if necessary---. \\

	Taking into account that the only property we know about the function $T$ is that it is Lipschitz, this distribution allows to center the sampling near the points in which the diagonal error is small. By the Lipschitz inequality $\|T(x)-T(a)\| \le K \|x-a\|,$ $T(x)$ and $T(a)$ are controled by the distance between $x$ and $a$, so using the proposed distribution  maximizes the probability of getting approximated eigenvectors in the sampling. 
	
\end{itemize}

The computations involved in the algorithm whose scheme has been presented above are easy to perform using R. To show concrete situations, we will show some numerical examples of functions $T:\mathbb R^2 \to \mathbb R^2,$ since they allow a graphical representation.
The algorithm used can be found in the Appendix \ref{app:R_code} of this work.

\begin{example}
Consider the parametric family of functions $R_r: \mathbb R^2 \to \mathbb R^2,$ $r \in \mathbb R,$ given by the expression 
$$
R_r(x,y)=\Big(    8x + r \cdot \sin(5xy),  4x^2 + 4xy + y^2 - 2x - \frac{1}{5} \, \sqrt( |x+y|)  \Big), \quad (x,y) \in \mathbb R^2.
$$
The calculations shown below have followed the next rule. We consider the domain subset $[-5,5] \times [-5,5].$
We start with $N= 500$ initial points in the domain, which are obtained randomly.  We choose the $N_0=100$ best with respect to the error value. 
Using the  distribution $\psi$ written in step (3) of the algorithm with $\tau=5$, for each of these points we generate $N_1=10$ points around, from which the best one is selected. We repeat this step $10$ times. 

Let us show the graphical representation of the results with three different values of the parameter $r.$

\begin{itemize}

\item $r=0.$  The sine term in the first coordinate is eliminated. This makes the example simpler, with an easy to understand representation of a suitable subset of approximate eigenvectors. 

\begin{figure}[htpb]
	\centering
	\includegraphics[width=0.45\linewidth]{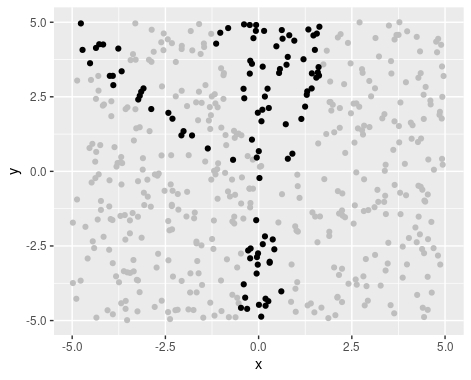}
	\hspace{0.3cm}
	\includegraphics[width=0.45\linewidth]{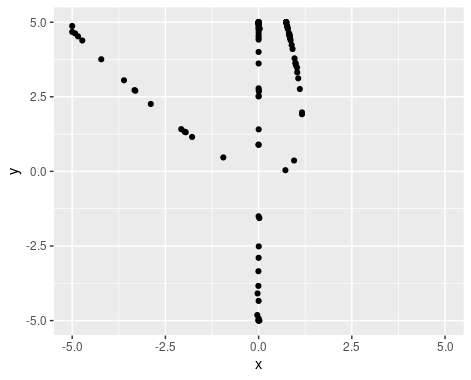}
	\caption{Left: first $500$ random points with the $100$ chosen points for $r=0$. Right: final set of approximated eigenvectors.}
	\label{fig:exampleini1}
\end{figure}

\item $r=3.$  In this case, the sine term in the first coordinate causes an important perturbation, producing a more dispersed eigenvector structure.

\begin{figure}[htpb]
	\centering
	\includegraphics[width=0.45\linewidth]{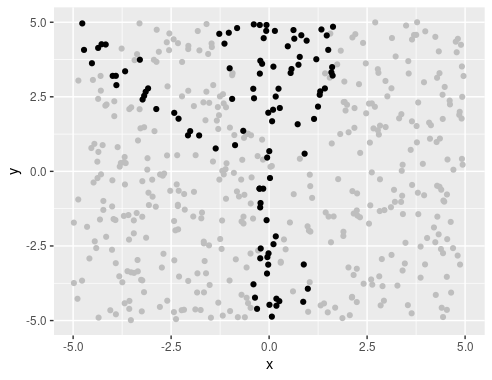}
	\hspace{0.3cm}
	\includegraphics[width=0.45\linewidth]{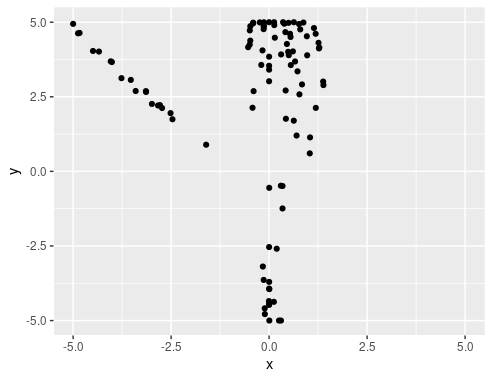}
	\caption{Left: first $500$ random points with the $100$ chosen points for $r=5$. Right: final set of approximated eigenvectors.}
	\label{fig:exampleini1}
\end{figure}

\item $r=-10.$  The sine term causes a stronger perturbation (of negative sign). The consequence is that the set of eigenvectors no longer follows (even approximately) clear lines.

\begin{figure}[htpb]
	\centering
	\includegraphics[width=0.45\linewidth]{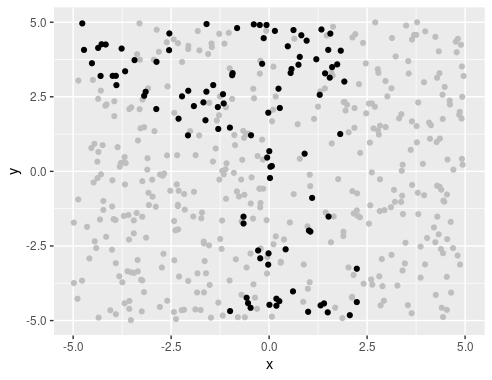}
	\hspace{0.3cm}
	\includegraphics[width=0.45\linewidth]{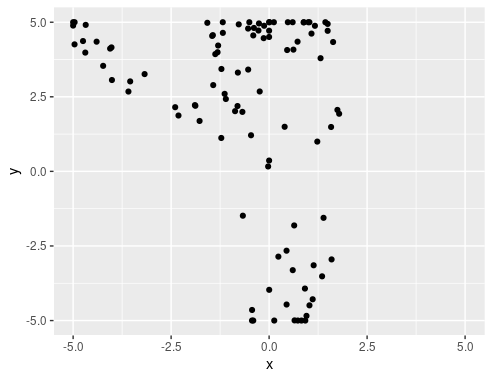}
	\caption{Left: first $500$ random points with the $100$ chosen points for $r=-10$. Right: final set of approximated eigenvectors.}
	\label{fig:exampleini1}
\end{figure}
\end{itemize}

\end{example}

These examples show that, although we cannot expect a diagonal distribution of eigenvectors for non lattice Lipschitz  maps, we can sometimes treat  general Lipschitz operators as  perturbations of lattice Lipschitz maps, in the case where some diagonal distribution of a set of approximate eigenvectors is still preserved. Since no clearly defined axes are given for such a set, we have to provide a technique for the generation of two straight lines (or $n$ straight lines in the general case) that can play this role. How to do so will be shown in the next section.

%
%
%

\section{General algorithm and examples}  \label{S4}

Let us show in this section how the extension/representation  formulas (McShane and Whitney versions) that work for the case of lattice Lipschitz operators can be adapted for any Lipschitz operator $T: S \subset E \to E$  in order to obtain an approximate functional expression for $T$ (Definition 2, Proposition 1 and  Proposition 2 in \cite{arn}). In the case of a lattice Lipschitz operator $L$, if we fix $S$ to be the union of the rays defined by a certain basis of eigenvectors of $L$ for $E,$ both of these formulas give exact representations of $L$ (Theorem 4 in \cite{arn}). That is, the operator $L$ coincides with both $(L|_S)^M$ and $(L|_S)^W.$ In case we consider $S$ an arbitrary subset of $E$, the McShane and Whitney formulas give approximations, for which the error expressions are known.
These adapted formulas are
$$
T^M(x)(w):=\bigvee \Big\{T(z)(w)- K(w) \vert x-z\vert (w): z \in S \Big\}, \quad x \in E,
$$
and
$$
T^W(x)(w):=\bigwedge\Big\{T(z)(w)+K(w) \vert x-z\vert (w): z \in S \Big\},  \quad x \in E.
$$
where each $w$ denotes the index on the corresponding element in the basis $\mathcal B,$ and the functions $K(w)$ is the pointwise evaluation of the Lipschitz constant for each coordinate.

Therefore, the use of these formulas
explicitly requires an order in space. Indeed, the expression $|s-v|(w)$ appearing in them is calculated using the order provided by a basis. In the next subsection we propose a method for defining such an order for the case of general Lipschitz maps.

\subsection{The definition of the order for the approximate representation of a Lipschitz operator}
\label{sec:order}

As we are designing an approximation method, the procedure  to find a good basis has to be related to the mass distribution of the set of approximate eigenvalues. The main idea consists in defining a partition of the set of approximate eigenvalues into $n$ sets that are intended to describe the mass distribution. The centres of mass of the subsets of the partition give the $n-$dimensional basis necessary for the definition of the lattice order.
Any clustering method could provide a technique for doing so. In this section we explain a method based on observing the mass distribution of the set of approximate eigenvectors obtained, from \textit{Principal Component Analysis} of the point cloud defined by this set. Fix an operator $T: \mathbb R^n \to \mathbb R^n.$
The proposed method follows the next steps, that give different solutions depending on the symmetry of the set of approximated eigenvectors  that are obtained.

\begin{itemize}
	
	\item[(1)]  The direct case: if  the  approximated eigenvectors are distributed around a set of $n$ vectors that are linearly independent, we take them as the adequate basis $\mathcal B.$
	
	\item[(2)] Otherwise,  we compute the PC (Principal Components) of the cloud of approximated eigenvectors of $T,$ that has been obtained. This technique is widely used for determined the main trends that can be detected in a point cloud (see for example \cite{abd,jol}).  We define the new (orthonormal) basis $\mathcal B$ using the PC.
	This is  a candidate for being a good basis in case the symmetry of the sets that define the distribution of mass of the eigenvectors coincide with the directions of the vectors provided by the PC.  The main problem is that this method always gives an orthogonal basis, which, as we have seen, is not necessarily the best way to describe the distribution of the eigenvector distribution, even if we have a lattice Lipschitz map.
	
	\item[(3)] Suppose now that the cloud of approximate eigenvectors is not oriented following the direction of any set of $n$ vectors, but can be found in a particular region of space. In this case we consider the octants (or hyperoctants) defined by the orthogonal basis provided by the PCA. We will choose a new basis $\mathcal C$ defined by the vectors crossing these octants along their axes of symmetry. That is,   take $\sigma$ to be any of the elements of $[-1,1]^n.$  We consider the vectors, expressed in the orthogonal basis provided by the PC, as
	$$
	c= \frac{\sigma}{\|\sigma\|} = \frac{ \big(\sigma(1),...,\sigma(n) \big)}{n^{1/2}}.
	$$
	For example, we get the first vector of  $\mathcal C$  to be $(1,1,1,...,1)/n^{1/2},$ and we  choose other $n-1$ vectors as the ones above to complete a basis. 
	
	
\end{itemize}

In the spirit of setting a concrete procedure for this article, we will follow the rule explained above for the definition of the order in the lattice in the next section. As we have said, this is not the only way that can be proposed to obtain such an order.  In general, any  rule for defining a suitable basis should depend on the symmetry of the problem.

\subsection{Weakening the lattice Lipschitz inequality}  \label{S4}

%
%
%
%
%
%
%
%
%

Although the definitions of the lattice versions of the McShane and Whitney extensions provide accurate results for diagonal Lipschitz maps, we cannot expect such good behavior for operators that are not exactly lattice Lipschitz.
If the mapping $T$ is not diagonalizable or the set of axes is not exactly determined, the assumptions of Theorem \ref{coorwise_iff} are not satisfied, so that $T$ may not be a lattice Lipschitz operator.
This makes that the associated function $K(w)$ may be too large, and therefore also the error of the approximation.\\

The solution we present for this problem is to change the condition \eqref{eq:lattice_lips_def} to
\begin{equation}
	\label{eq:lips_lattice_norm}
	\big| T(x) - T(y) \big|(i) \leq K(i) \big( (1 - \alpha) | x - y |(i) + \alpha \cdot || x - y ||  \big) , \quad 1 \leq i \leq n,
\end{equation}
where $0 \leq \alpha \leq 1$ and $|| \cdot ||$ is a norm on $\mathbb R^n$.
Writing it in terms of the order of the space $E,$ as 
\begin{equation*}
	| T(x) - T(y) | \leq K \big( (1-\alpha) | x - y | + \alpha || x - y || \cdot \textbf{1} \big).
\end{equation*}
where $\bf{1}$ denotes the constant function one in $\Omega$.\\

Observe that if $\alpha = 0$, the condition \eqref{eq:lips_lattice_norm} is the same as \eqref{eq:lattice_lips_def} and, if $\alpha = 1$, it is equivalent to every  coordinate function $T_i$  of $T$  being a real Lipschitz function.
In the case that $T = L + P$ is the addition of a lattice Lipschitz operator $L$ plus a perturbation $P$, which we assume to have a small Lipschitz constant, $T$ satisfies
\begin{equation*}
	| T(x) - T(y) | \leq \left( \frac{K}{1-\alpha} + \frac{C}{\alpha} \right) \cdot \big( (1-\alpha) | x - y | + \alpha || x - y || \cdot \textbf{1} \big),
\end{equation*}
where $K$ is the associated function of $L$ and $C$ the Lipschitz constant of $C$. Note that, to control the function $K$ of \eqref{eq:lattice_lips_def}, $\alpha$ can be smaller the smaller is the perturbation $P$.

If a function defined on a subset $S$ of $\mathbb R^n$, $T : S \to \mathbb R^n$ satisfies \eqref{eq:lips_lattice_norm}, we can also define the McShane and Whitney extensions as
\begin{align*}
	T^M(x)(i) & := \bigvee \big\{ T(z)(i) - K(i) \big( (1-\alpha) |x-z|(i) + \alpha || x - z || \big) : z \in S \big\}, \\
	T^W(x)(i) & := \bigwedge \big\{ T(z)(i) + K(i) \big( (1-\alpha) |x-z|(i) + \alpha || x - z || \big) : z \in S \big\}.
\end{align*}
Following the arguments presented in \cite{arn}, we find that the error bounds become in this case
\begin{equation}
	\label{eq:bounds_with_alpha}
	\begin{split}
		- 2 K \bigwedge \{ (1-\alpha) | x - z | + \alpha || x - z || \cdot \textbf{1} : \, z \in E_0 \} \leq (T|_{E_0})^M(x) - T(x) \leq 0 \\
		0 \leq (T|_{E_0})^W(x) - T(x) \leq 2 K \bigwedge \{ (1-\alpha) | x - z | + \alpha || x - z || \cdot \textbf{1} : \, z \in E_0 \}.
	\end{split}
\end{equation}
The idea now is to use the previous method to approximate a function, using for doing so a small $\alpha > 0$.

\subsection{General algorithm} \label{S4.3}
Using the formulas explained, we apply the complete procedure following the steps below.

\begin{itemize}
	
	\item[(1)] Fix an operator $T: \mathbb{R}^n \to \mathbb{R}^n$ that one wants to reconstruct from its eigenvectors. If the set of eigenvectors is known, sample some points in it, calculate the value of $T$ and go to the next step. Otherwise, the eigenvector set is approximated by a Monte Carlo procedure based on diagonal error minimization as is explained in Section \ref{sec:approx_eigenvectors}. Call $S$ the set of ---approximate--- eigenvectors.
	
	\item[(2)] For fixing the order on $\mathbb R^n$, use the method exposed on Section \ref{sec:order} based on PCA ---other procedures could also be possible---.
	
	\item[(3)] Fix a small $\alpha > 0$ (for example $\alpha = 0.1$) and compute the best $K(w)$ possible by using the formula
	\begin{equation*}
		K(w) = \max \left\{ \dfrac{ | T(x) - T(y) |(w) }{ (1 - \alpha) | x - y |(i) + \alpha \cdot || x - y || } : \, x, y \in S, \, x \neq y \right\},
	\end{equation*}
	and the McShane and Whitney lattice-type formulas associated to the order provided in the previous step. 
	
	\item[(4)] An interpolation of these formulas provide a(n approximate) representation of the original operator $T.$ We control the error commited by these formulas by using the error associated to the Lipschitz inequality and the formulas in (\ref{eq:bounds_with_alpha}), together with the diagonal error of the approximation of the set of eigenvectors. 
	
\end{itemize}

\section{A numerical example} \label{S5}

Let us show how the  method explained in Section \ref{S4}   works in a concrete numerical example. Let $E = \mathbb R^2$ and consider the function $f: \mathbb R^2 \to \mathbb R^2$ defined as the function  $(x,y) \mapsto (x+y, x-y)$ (which is a diagonal map, and a lattice Lipschitz operator), with a small perturbation: 
\begin{equation*}
	f(x,y) = \left( x + y + \frac{1}{5} \sin(10x) + \frac{xy}{100}, x - y - \frac{1}{10} \cos(x+5y) \right).
\end{equation*}

First of all, we approximate the set of eigenvectors with the Monte Carlo method explained above. Let $S_0$ be a sample $250$ points using a uniform distribution on the set $P = [-5,5] \times [-5,5]$. Compute the diagonal error $\varepsilon(\cdot)$ at each point of $S_0$, and select the $50$ of them with smallest error ($20\%$ of the points on $S_0$). Write $S_1$ for this set. Now, for each element $x$ on $S_1$, we start an iterative process to find a set of approximated eigenvectors as  explained before. After $5$ iterations, the result (the set $S_2$) is plotted in Figure \ref{fig:example_1_S}.

\begin{figure}[htpb]
	\centering
	\includegraphics[width=0.45\linewidth]{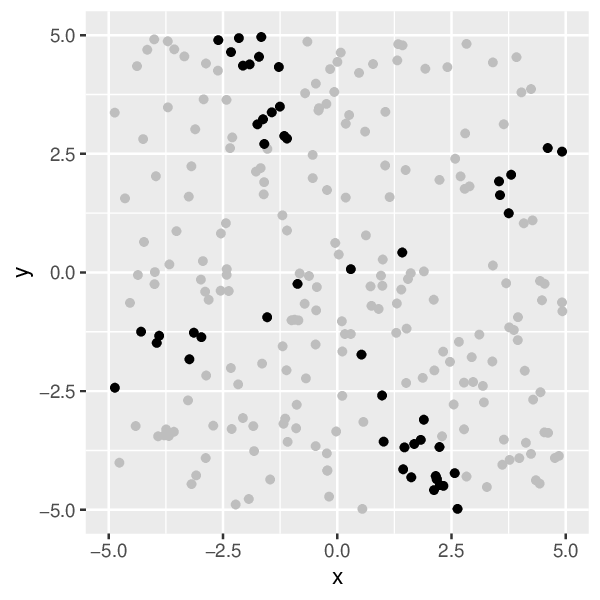}
	\hspace{0.4cm}
	\includegraphics[width=0.45\linewidth]{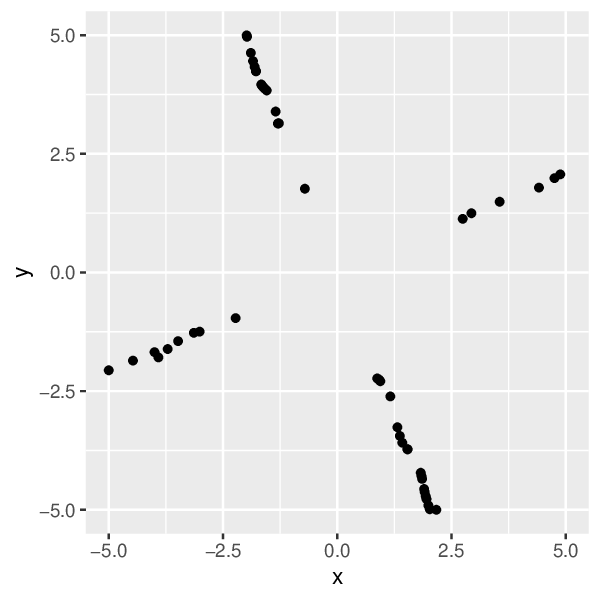}
	\caption{On the left, first approach to the set of eigenvectors, $S_0$ in gray and $S_1$ in black.  The set $S_2$ obtained after $5$ iterations can be seen on the right.}
	\label{fig:example_1_S}
\end{figure}

In the next step we choose the axes that will define the lattice structure of $E$. In this case, we apply   Principal Component Analysis (PCA)  \cite{abd}. The resulting new axes can be seen in  Figure \ref{fig:example_1_PCA}. These axes "look good" because they have a distribution similar to that of the true eigenvectors.

\begin{figure}[htpb]
	\centering
	\includegraphics[width=0.5\linewidth]{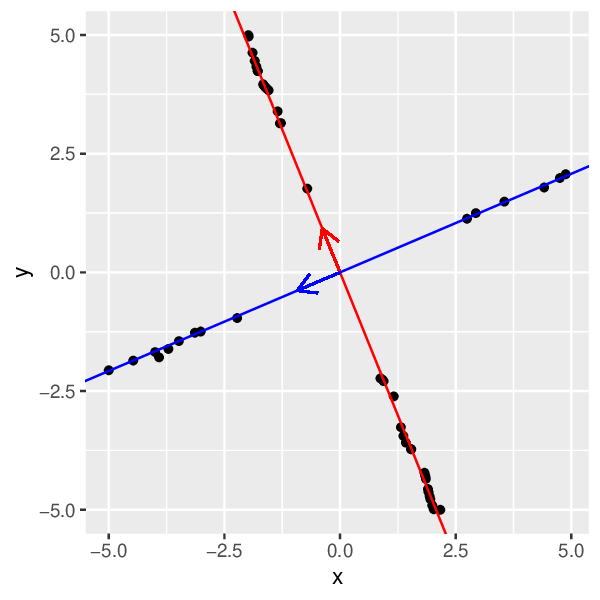}
	\caption{The new axes resulting from PCA; in red, the first principal component and, in blue, the second one.}
	\label{fig:example_1_PCA}
\end{figure}

The final step is to compute the McShane and Whitney formulas (with $\alpha = 0.1$ and the Euclidean norm $|| \cdot ||$), using the points of $S_2$ and the order given by the orthonormal basis provided by the PCA. The best function $K$  is in this case
\begin{equation*}
	K(1) = 2.24, \quad K(2) = 3.26.
\end{equation*}
Observe that if the norm modification is not considered on the lattice Lipschitz inequality ($\alpha = 0$), the best $K$ possible is much larger: $K(1) = 16.2, K(2) = 220.4$, which would cause a worst approximation.
The approximation result computed as the mean value of the McShane and Whitney formulas,
\begin{equation*}
	\widehat{f}(x,y) = \dfrac{f^M(x,y) + f^W(x,y)}{2},
\end{equation*}
can be seen in Figure \ref{fig:example_1_ap}.

\begin{figure}[htpb]
	\centering
	\includegraphics[width=0.35\linewidth]{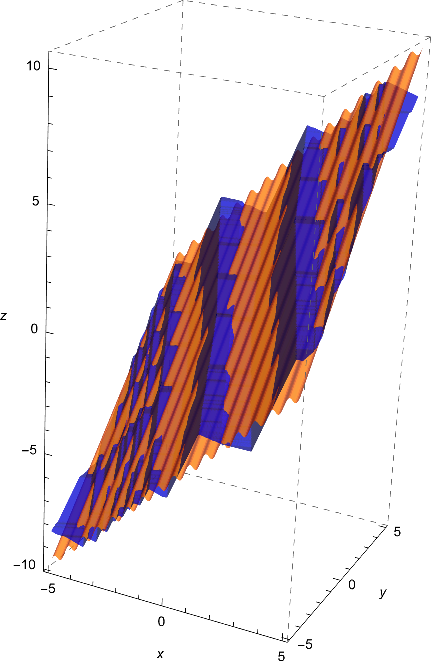}
	\hspace{1.0cm}
	\includegraphics[width=0.35\linewidth]{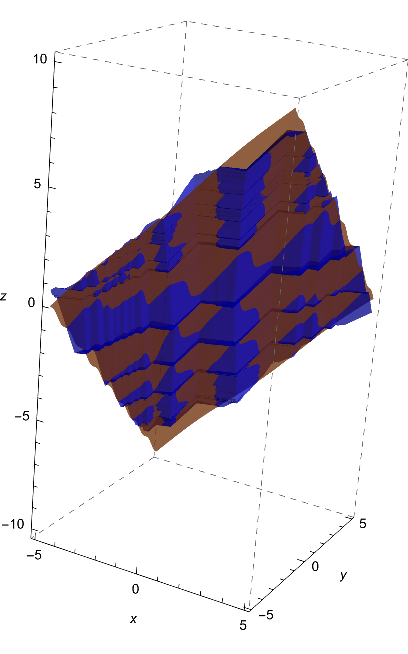}
	\caption{Approximation $\widehat{f}$ (in blue) of $f$ (in orange)---left:  first component, right: the second one---.}
	\label{fig:example_1_ap}
\end{figure}

In order to compare our approximation with the original $f$, we compute the error using a Monte-Carlo procedure,  and we obtain
\begin{equation*}
	\frac{1}{\mu(P)} \big\| \big( \| \widehat{f} - f \| \big) \big\|_2
	= \frac{1}{100} \left( \int_P \| \widehat{f} - f \|^2 dx \right)^\frac{1}{2}
	\approx 0.65.
\end{equation*}

The pointwise error is bounded; using the formulas  \eqref{eq:bounds_with_alpha} we obtain
\begin{align*}
	\big| \widehat{f}(x,y) - f(x,y) \big|(i) \leq K(i) \bigwedge_{(z,t) \in S_2} \Big( 0.9 (|x-z|,|y-t|) (i) + 0.1 \| (x,y) - (z,t) \|_2 \Big),
\end{align*}
for each component, $i = 1, 2$.

\section{Conclusions} \label{con}

The notion of lattice Lipschitz operator in finite-dimensional normed spaces has been introduced to provide a suitable set of Lipschitz-type operators that can be used for the design of approximation algorithms. Since any lattice Lipschitz operator always allows a diagonal representation, the family of functions to which this approximation method is applied is composed of nonlinear functions with the property that they allow an ``almost diagonal'' representation.

This makes it necessary to find an approximation method to find the ``almost eigenvalues'' of the objective function, and, in a second step, to determine a good set of lattice Lipschitz maps that can be used as an approximation family for it.  We propose a concrete algorithm, and show with an example how it works, taking into account the measure of the error made when using the approximation, whose formulas have also been obtained in the paper.


\appendix

\section{R code}
\label{app:R_code}

\begin{lstlisting}{}

library(ggplot2)
library(tidyverse)

ev_lambda <- function(v, f){
  la = ifelse(all(v==c(0,0)), 0, sum(f(v)*v)/sum(v^2) )
  return(la)
}

ev_error <- function(v, f){
  ev_la <- ev_lambda(v, f)
  if(all(v==c(0,0))){
	return(0)
  }else{
	return(sqrt(sum((f(v)-ev_la*v)^2))/sqrt(sum(v^2)))
  }
}

aprox_eigen_R2 <- function(FUN, range = 1, N, N0, N1, tau, steps = 1){
  # FUN: the original function
  # range: work on the rectange [-range,range]^2
  # N: number of points to start the first time
  # N0: number of points (of the N) to select
  # N1: number of points to generate
  #   of each of N0 to select the best
  # tau: the factor of the variance for the gaussian
  # steps: steps to repeat
  # output: a list with all the information
  
  res <- list(S=list())
  
  # STEP 0
  df <- data.frame(
	x = runif(N, min=-range, max=range),
	y = runif(N, min=-range, max=range)
  ) %>% 
	mutate(ev_err=map2_dbl(x, y,
		function(x,y){ ev_error(c(x,y), FUN)})) %>% 
    arrange( ev_err )
  res$S00 <- df
  
  res$plot_step_0 <- ggplot() +
    geom_point(data=df, mapping=aes(x, y), color="gray") +
    geom_point(data=df[1:N0,], mapping=aes(x, y), color="black") +
    coord_cartesian(xlim=c(-range, range), ylim=c(-range, range)) 
  
  df <- df[1:N0,]
  res$S0 <- df
  
  # STEPS >= 1
  for(s in 1:steps){
	for(i in 1:N0){
	  x = c(df$x[i], rnorm(
	    N1, mean=df$x[i], sd=tau*df$ev_err[i]))
	  y = c(df$y[i], rnorm(
	    N1, mean=df$y[i], sd=tau*df$ev_err[i]))
	  ev_err <- map2_dbl(x, y,
	    function(x,y){ev_error(c(x,y),FUN)})
	  ind <- ev_err %>% which.min()
	  if(x[ind]>range){x[ind]=range}
	  if(x[ind]<-range){x[ind]=-range}
	  if(y[ind]>range){y[ind]=range}
	  if(y[ind]<-range){y[ind]=-range}
	  df[i,] <- c(x[ind], y[ind], ev_err[ind]) %>%
	  data.frame() %>% t()
    }
    res$S[[s]] <- df
  }
	
  res$eigenvec <- df %>%
  select(x, y) %>% 
  mutate(lambda = map2_dbl(x, y, ~ev_lambda(c(.x,.y), FUN)))
  res$mean_err <- mean(df[, 'ev_err'])
  res$plot_step_end <- ggplot() +
  geom_point(data=df, mapping=aes(x, y), color="black") +
  coord_cartesian(xlim=c(-range, range), ylim=c(-range,range)) 
  return(res)
  
}

### EXAMPLE ------ REF ------

R0 <- function(v) {
  x = v[1]
  y = v[2]
  return( c(8*x, 4*x^2 + 4*x*y +y^2 - 2*x - 1/5*sqrt(abs(x+y))) )
}

ev <- aprox_eigen_R2(FUN = R0, range = 5, N = 500, N0 = 100,
                     N1 = 10, tau = 5, steps = 10)
ev$plot_step_end

\end{lstlisting}

\vspace{1cm}

\textbf{Funding:} The first author was supported by a contract of the Programa de Ayudas de Investigación
y Desarrollo (PAID-01-21), Universitat Politècnica de València.

\vspace{1cm}





\end{document}